\documentclass[10pt]{article}
\usepackage{amssymb,amsfonts}

\usepackage[dvips]{graphicx}
\usepackage{psfrag}
\usepackage{ifthen}
\usepackage[usenames]{color}
\usepackage{caption}
\usepackage{pinlabel}

\usepackage{amssymb}
\usepackage{latexsym}
\usepackage{amsmath}
\usepackage{mathrsfs}
\usepackage[all]{xy}
\usepackage{enumerate}
\usepackage{amsthm}
\usepackage{amsmath,amsthm,amsfonts,amssymb}
\usepackage[usenames]{color}

\usepackage{amsthm}

\makeatletter
\newtheorem*{rep@theorem}{\rep@title}
\newcommand{\newreptheorem}[2]{%
\newenvironment{rep#1}[1]{%
\def\rep@title{#2 \ref{##1}}%
\begin{rep@theorem}}%
{\end{rep@theorem}}}
\makeatother

\newreptheorem{theorem}{Theorem}

\newtheorem{theorem}{Theorem}
\newtheorem{definition}{Definition}
\newtheorem{example}{Example}
\newtheorem{lemma}{Lemma}

\newtheorem{remark}{Remark}

\newtheorem{prop}{Proposition}

\newcommand{\be}{\begin{enumerate}}
\newcommand{\ee}{\end{enumerate}}
\newcommand{\beq}{\begin{equation}}
\newcommand{\eeq}{\end{equation}}
\usepackage{amssymb}

\usepackage{amsmath}
\usepackage{amssymb}

\newcommand{\gst}{\, | \,} % "group such that", vertical bar with small space, for group presentations
\newcommand{\bel}[1]{\begin{equation}\label{#1}}
\newcommand{\eee}{\end{equation}}

\newcommand{\LBA}{\left\{\begin{array}}
\newcommand{\EAR}{\end{array}\right.}
\input epsf

\newenvironment{romanenumerate}
{\begin{enumerate}

}
{

\end{enumerate}
} % enumerate list with roman labels
\newcommand{\comment}[1]{} % for multi-line comments
\begin{document}
\title{Effective construction of covers of canonical Hom-diagrams  \\for equations over torsion-free hyperbolic groups}
\author{Olga Kharlampovich, Alexei Myasnikov, Alexander Taam}

\maketitle
\begin{abstract}
We show that, given a finitely generated group $G$ as the coordinate group of a finite system of equations over a torsion-free hyperbolic group $\Gamma$, there is an algorithm which constructs a cover of a canonical solution diagram. The diagram encodes all homomorphisms from $G$ to $\Gamma$ as compositions of factorizations through $\Gamma$-NTQ groups and canonical automorphisms  of the corresponding NTQ-subgroups.  We also give another characterization of  $\Gamma$-limit groups as iterated generalized doubles over $\Gamma$.
\end{abstract}

\section{Introduction}
Given a group $G$, another group $L$ is said to be \emph{fully residually $G$} (or \emph{discriminated by $G$}) if, given any finite subset $A\subset L$, $1\notin A$, there exists a homomorphism $\phi:L\rightarrow G$, such that $\phi(a)\neq 1$ for all $a\in A$. The study of groups discriminated by various classes of groups has been ongoing since the early twentieth century, with increased activity often occurring as other equivalent characterizations of such groups, and applications in different fields, are found. For example, the development of algebraic geometry over groups allowed finitely generated groups discriminated by a free group $F$ to be understood as coordinate groups of irreducible systems of equations over $F$. This helped enable solutions to Tarski's problems on the elementary theory of free groups, given independently by Kharlampovich-Myasnikov and Sela (\cite{elemFree} and \cite{selaTarski}).

Throughout, we let $\Gamma$ be an arbitrary fixed torsion-free hyperbolic group. While some of the theory of fully residually free groups can be generalized easily to groups discriminated by $\Gamma$, in general groups discriminated by $\Gamma$ are more difficult to work with for several reasons. Unlike finitely generated fully residually free groups, finitely generated fully residually $\Gamma$ groups may be not finitely presented and can contain finitely generated subgroups with no finite presentation, so algorithmic results are often  more difficult to obtain.

The following  characterizations of $\Gamma$-limit groups are well known. Let $L$ be a finitely generated group, then the following are equivalent.
\begin{enumerate}
\item $L$ is a $\Gamma$-limit group (in the space of marked groups, see \cite{Gri84} for definition)
\item $L$ is fully residually $\Gamma$
\item $L$ embeds in an ultrapower of $\Gamma$
\item $Th_{\forall}(\Gamma)\subseteq Th_{\forall}(L)$ ($Th_{\forall}$ is universal first order theory)
\item $L$ is the coordinate group $\Gamma _{R(S)}$ (see Section \ref{section:alg-geo}) of an irreducible (in the Zariski topology) algebraic set over $\Gamma$ defined by a system $S$ of coefficient-free equations (call such a system an \emph{irreducible system})
\end{enumerate}

There are analogous equivalent characterizations for the case with coefficients (see \cite{KMlyndon} Theorem B). 

Given a group $G$, a free rank one extension of centralizer over $G$ is a group $E=\langle G,t|[C_G(U),t]=1\rangle$, where $U\subseteq G$ is finite. Note that $E\cong G*_{C_G(U)}\cong G*_{C_G(U)}(C_G(U)\times \mathbb{Z})$. A free rank $n$ extension of centralizer is defined similarly as $G*_{C_G(U)}(C_G(U)\times \mathbb{Z}^n)$. A group is called an \emph{iterated extension of centralizer over $G$} if it may be obtained from $G$ by a finite sequence of extension of centralizers (each centralizer in the previous group).

\begin{prop}\label{prop:iter-ext}(Theorem E in \cite{KMlyndon})
A group $L$ is a $\Gamma$-limit group if and only if it embeds into an iterated extension of centralizers over $\Gamma$.
\end{prop}

In fact, the embeddings into iterated extensions of centralizers are obtained via embeddings into $\Gamma$-NTQ groups, which are coordinate groups of nicely structured systems of equations (see Section \ref{sec:NTQ}). Some relevant details of such embeddings are described in Section \ref{section:homoms}, along with other related constructions.

Our first main theorem shows that, for a fixed torsion-free hyperbolic group $\Gamma$, the class of iterated generalized doubles over $\Gamma$ (terminology introduced by Champetier and Guirardel in \cite{markgrps} for free groups) is also equivalent to the class of $\Gamma$-limit groups.

\begin{definition}A group $G$ is a \emph{generalized double} over a $\Gamma$-limit group $L$ if it splits as $G=A*_CB$, or $G=A*_C$, where $A$ and $B$ are finitely generated such that:
\begin{enumerate}
\item $C$ is a non-trivial abelian group whose images are maximal abelian subgroups in the vertex groups.
\item there is an epimorphism $\phi:G\rightarrow L$ which is injective on each vertex group
\end{enumerate}
\end{definition}
Note that each vertex group is discriminated by $\Gamma$.
\begin{definition}
A group is an \emph{iterated generalized double over $\Gamma$} if it belongs to the smallest class of groups $\mathcal{IGD}-\Gamma$ that contains  groups isomorphic to subgroups of $\Gamma$, and is stable under free products and the construction of generalized doubles over groups in $\mathcal{IGD}-\Gamma$.
\end{definition}
Champetier and Guirardel showed in \cite{markgrps} that $\mathcal{IGD}-F$ is the class of $F$-limit groups. Our first main theorem shows that the same holds for $\Gamma$-limit groups.

\begin{theorem}\label{Gamma-IGD}
A group  is a $\Gamma$-limit group if and only if it is an iterated generalized double over $\Gamma$ 
\end{theorem}
Theorem \ref{Gamma-IGD} is proved in Section \ref{sec:IGD}. The proof uses Bass-Serre theory and Proposition \ref{prop:iter-ext}.

Recall that the Grushko decomposition of a finitely generated group $G$ is the free product decomposition $G=F_r*A_1*\cdots A_k$, where $F_r$ is a free group of finite rank, and each $A_i$ is non-trivial, freely indecomposable, and not infinite cyclic (see \cite{Gr40}). This decomposition is unique up to permutation of the conjugacy classes of the $A_i$ in $G$.  
 
Our second main result improves on the results of \cite{embedExt} which is necessary for the proof of the decidability of the $\forall\exists$-theory and elementary theory of a torsion-free hyperbolic group in the subsequent paper \cite{elemHyp}. Given a finite system of equations over $\Gamma$, a canonical Hom-diagram  (a.k.a. Makanin-Razborov diagram)  encodes all homomorphisms from the coordinate group $\Gamma_{R(S)}$ to $\Gamma$ as compositions of factorizations through $\Gamma$-limit quotients, and  canonical automorphisms of those quotients which correspond to JSJ decompositions of freely indecomposable factors (of those quotients).  Furthermore certain  $\Gamma$-NTQ groups (defined in Section \ref{sec:NTQ}) and a canonical completed Hom-diagram  are associated to such a diagram.  We algorithmically construct corrective extensions of these  $\Gamma$-NTQ groups (defined in Section \ref{subsection:IFT}) that form a cover of some  canonical completed Hom-diagram for $\Gamma_{R(S)}$, which is used in the subsequent paper \cite{elemHyp}. In (\cite{embedExt}, lemma 3.15) a similar diagram was algorithmically  constructed  but it was not canonical in a sense that automorphisms did not correspond to JSJ decompositions of $\Gamma$-limit quotients, they corresponded to some abelian splittings of these quotients.  We show that by eliminating some branches of this tree it becomes a cover of a canonical completed Hom-diagram.

\begin{theorem}  \label{thm:can-construct}Let $S(Z, A)=1$ be a finite system of equations over $\Gamma$. There is an algorithm to construct  a  cover (Definition \ref{cover}) of a   canonical completed Hom-diagram   for  the coordinate group $\Gamma _{R(S)}$.  \end{theorem}

Theorem \ref{thm:can-construct} is  proved in Section \ref{section:can-construct} using   the so-called Implicit Function Theorem (Proposition \ref{IFT}).

\section{Preliminaries}\label{section:prelims}
\subsection{Algebraic geometry over groups}\label{section:alg-geo}
We start with some basic theory of algebraic geometry over groups. For more background, see \cite{BMR99}. Let $G$ be a group generated by a finite set $A$ and $F(X)$ the free group on $X=\{x_1,\ldots,x_n\}$. For $S\subset G[X]=G*F(X)$, the expression $S(X,A)=1$ is called a \emph{system of equations} over $G$, and a \emph{solution} of $S(X,A)=1$ in $G$, is a $G$-homomorphism $\phi:G[X]\rightarrow G$ such that $\phi(S)=1$ (a $G$-homomorphism is determined by $Z=\phi(X)\in G^n$; the notation $S(Z,A)=1$ means that $Z$ corresponds to a solution of $S$). 

Denote the set of all solutions of $S(X,A)=1$ in $G$ by $V_G(S)$, the \emph{algebraic set defined by $S$}. Define the Zariski topology on $G^n$ by taking algebraic sets as a pre-basis of closed sets. Note that the algebraic set $V_G(S)$ uniquely corresponds to the normal subgroup $R(S)=\{T(X,A)\in G[X]|\forall Z\in G^n(S(Z,A)=1\rightarrow T(Z,A)=1)\}$ in $G[X]$, called the \emph{radical} of $S$. Call $G_{R(S)}=G[X]/R(S)$ the \emph{coordinate group} of $S$. Every solution of $S(X,A)=1$ in $G$ can be described as a $G$-homomorphism $G_{R(S)}\rightarrow G$.
\subsubsection{Quadratic equations}
A system of equations $S(X,A)=1$ is said to be (strictly) \emph{quadratic} if each $x_i\in X$ that appears in $S$, appears at most (exactly) twice, where $x_i^{-1}$ also counts as an appearance. There are four standard quadratic equations:
\begin{equation}\label{closedOrientable}\prod_{i=1}^n[x_i,y_i]=1; n\geq 1\end{equation}
\begin{equation}\label{boundOrientable}\prod_{i=1}^n[x_i,y_i]\prod_{i=1}^mc_i^{z_i}d=1; n,m\geq 0,n+m\geq 1\end{equation}
\begin{equation}\label{closedNonOrientable}\prod_{i=1}^nx_i^2=1; n\geq 1\end{equation}
\begin{equation}\label{boundNonOrientable}\prod_{i=1}^nx_i^2\prod_{i=1}^mc_i^{z_i}d=1; n,m\geq 0,n+m\geq 1\end{equation}
where $c_i$ and $d$ are non-trivial elements of $G$. Note that for any strictly quadratic word $S\in G[X]$, there is a $G$-automorphism of $G[X]$ that takes $S$ to a standard quadratic word. Also, there is a surface with boundary associated to each standard quadratic equation, specifically an orientable surface of genus $n$ with zero punctures for (\ref{closedOrientable}), an orientable surface of genus $n$ with $m+1$ punctures for (\ref{boundOrientable}), a non-orientable surface of genus $n$ and zero punctures for (\ref{closedNonOrientable}), and a non-orientable surface of genus $n$ with $m+1$ punctures for (\ref{boundNonOrientable}). For a standard quadratic equation $S$, let $\chi(S)$ be the Euler characteristic of the associated surface.

\subsubsection{$G$-NTQ groups}\label{sec:NTQ}
General systems of equations can exhibit some properties similar to those of quadratic systems. This can be seen when systems are in a particular \emph{quasi-quadratic} form.

\begin{definition}\label{NTQ}
A system of equations $S(X,A)=1$ over a non-abelian group $G$ generated by $A$, is called \emph{triangular quasi-quadratic over $G$} or $G$-TQ, if it can be partitioned into subsystems:
$$S_1(X_1,C_1)=1$$
$$S_2(X_2,C_2)=1$$
$$\vdots$$
$$S_n(X_n,C_n)=1$$
where:
\begin{enumerate}[(i)]
\item $\{X_1,\ldots,X_n\}$ is a partition of $X$;
\item $G_i=G[X_i,\ldots,X_n, A]/R_G(S_i,\ldots,S_n)$ for $1\leq i \leq n$, $G_i$ is the group corresponding to level $i$;
\item $G_{n+1}=G\ast F$ (where $F$ is a free group of finite rank) or a subgroup of this group that is a free product of $F_1$, $G$ and  conjugates of $G$ by  generators of $F_2$, where $F=F_1\ast F_2$;
\item $C_i=X_{i+1}\cup\ldots\cup X_n\cup A\subset G_{i+1}$ for $1\leq i \leq n-1$ and $C_n=A$.
\end{enumerate}

Furthermore, for each $i$ the subsystems $S_i$ must have one of the following forms:
\begin{enumerate}[(I)]
\item $S_i$ is quadratic in $X_i$
\item $S_i=\{[x,y]=1,[x,u]=1|x,y\in X_i, u\in U\}$ where $U$ is a maximal abelian subgroup of $F(X_{i+1},\ldots,X_n,A)$
\item $S_i=\{[x,y]=1|x,y\in X_i\}$
\item $S_i$ is empty
\end{enumerate}
The number $n$ is called the \emph{depth} of the system.

Very often when $C_{i}=C_{i+1}=\ldots =C_{i+j}$ we will join several equations $S_i(X_i,C_i)=1,\ldots S_{i+j}(X_{i+j}, C_i)=1$ into one system $T(X_i,\ldots, X_j,C_i)=1$.
Then the same system $S(X,A)=1$  will be represented as the union
$$T_1(\bar X_1,\bar C_1)=1$$
$$T_2(\bar X_2,\bar C_2)=1$$
$$\vdots$$
$$T_m(\bar X_m,\bar C_m)=1,$$
where $X$ is a disjoint union of $\bar X_1,\ldots ,\bar X_m$ and the group of level $i$ is defined as $G[\bar X_i,\ldots,\bar X_m, A]/R_G(T_i,\ldots,T_m)$.

\end{definition}

Notice that it may be assumed that every subsystem of form (I) consists of  quadratic equations in standard form. Also, it can be checked directly that $G_i\simeq G_{i+1}[X_i]/R_{G_{i+1}}(S_i)$. $S(X,A)=1$ is called \emph{non-degenerate triangular quasi-quadratic over $G$} or $G$-NTQ if it is $G$-TQ and for every $i$, the system $S_i(X_i,C_i)=1$ has a solution in $G_{i+1}$, and if $S_i$ is of form (II) the set $U$ generates a centralizer in $G_{i+1}$. A \emph{regular} $G$-NTQ system is a $G$-NTQ system in which each non-empty quadratic equation $S_i$ is in standard form, and either $\chi(S_i)\leq -2$ and the quadratic equation has a non-commutative solution in $G_{i+1}$, or it is an equation of the form $[x,y]d=1$ or $[x_1,y_1][x_2,y_2]=1$.

Finally a group is called a \emph{(regular) $G$-NTQ group} if it is isomorphic to the coordinate group of a (regular) $G$-NTQ system of equations. Note that in Sela's work \cite{selaTarski}, \emph{$\omega$-residually free towers} provide an analogous structure to $F$-NTQ groups, and in the work of Casals-Ruiz and Kazachkov \cite{CRK}, \emph{graph towers} are, in a certain sense, higher dimensional analogues of NTQ systems for working with right angle Artin groups. 

 Every $\Gamma$-NTQ group  is toral relatively hyperbolic \cite{KMlyndon}. 

\subsection{Graphs of groups}\label{subsec:graphs-grps}
A \emph{graph of groups} $\mathcal{G}(X)$ is a connected graph $X(V,E)$ labeled with a group $G_v$ for each vertex $v\in V$, and a group $G_e$ with monomorphisms $\alpha_e:G_e\rightarrow G_{\partial_0(e)}$, $\beta_e:G_e\rightarrow G_{\partial_1(e)}$ for each edge $e\in E$ ($\partial_0(e)$ and $\partial_1(e)$ denote the initial and terminal vertices of $e$ respectively). Note that $X$ is considered to be a non-oriented graph, (i.e. there is an involution $\bar{}: E\to E$ with $\partial_0(e)=\partial_1(\bar e)$ for each $e\in E$), so the monomorphisms of a graph of groups must also satisfy $\alpha_e(G_e)=\beta_{\bar e}(G_{\bar e})$.

Let $T$ be a maximal subtree of $X$. The \emph{fundamental group of $X(V,E)$} with respect to $T$ is the group $\pi(\mathcal{G}(X),T)$ which is generated by $ \langle*_{v\in V}G_v,*_{e\in E}t_e\rangle$ with relations $\{t_e=1\forall e\in T,t_et_{\bar e}=1\forall e\in E,t_{\bar e}\alpha(g)t_e=\beta(g)\forall g\in G_e,\forall e\in E\}$. Any choice of maximal subtree gives an isomorphic fundamental group of the graph of groups.
A \emph{splitting} of a group $G$ over some class of group $\mathcal{E}$ is an isomorphism from $G$ to $\pi(\mathcal{G}(X(V,E)),T)$ where each $G_e$ is in $\mathcal{E}$. Splittings are discussed in further detail in Section \ref{section:canonical-splittings}. 

Graphs of groups are closely related to the Bass-Serre theory of groups acting without inversion on trees. We note here one significant consequence which will be used later. 

\begin{prop}\label{prop:BSsubgrp}(see \cite{cohen}, Theorem 3.7 in \cite{SW79})
Given $G$ the fundamental group of a graph of groups, and $H\leq G$. Then $H$ is the fundamental group of the induced graph of groups with vertex and edge groups given by intersections of $H$ with conjugates of vertex and edge groups of $G$, respectively.
\end{prop}

\section{Iterated generalized doubles}\label{sec:IGD}
We can now prove Theorem \ref{Gamma-IGD}. We start with the following useful elementary example of a generalized double.

\begin{lemma}\label{extCent}
A free rank one extension of centralizer over $\Gamma_1$, where $\Gamma_1\leq\Gamma$, is a generalized double over $\Gamma$.
\end{lemma}
\begin{proof}
Given $G=<\Gamma _1,t|[C_{\Gamma_1}(U),t]=1>=\Gamma_1*_{C_{\Gamma_1}(U)}$, let $\phi:G\rightarrow\Gamma_1$
\begin{displaymath}
\phi(x) = \left\{
\begin{array}{lr}
x &; x \in \Gamma_1\\
1 &; x=t
\end{array}
\right.
\end{displaymath}

$\phi$ is injective on $\Gamma_1$ and by definition a centralizer is maximal abelian in $\Gamma _1$.
\end{proof}

With Proposition \ref{prop:iter-ext}, this is enough to prove one direction of Theorem \ref{Gamma-IGD}.

\begin{prop}Let $G$ be a generalized double over a $\Gamma$-limit group $L$. Then $G$ is a $\Gamma$-limit group.
\end{prop}
\begin{proof}First consider the case where $G$ is an amalgamated product. Let $\phi$ be the map from $G$ to $L$ as in the definition of generalized double. Denote the images of $A,B,C$ under $\phi$ by $A^{\phi}$, $B^{\phi}$, and $C^{\phi}$ respectively. Let $\hat{C}$ be the maximal abelian subgroup of $L$ containing $C^{\phi}$. Consider the group $\hat{L}=<L,t|[c,t]=1,c\in\hat{C}>$. Since $L$ embeds into an iterated centralizer extension of $\Gamma$, so does $\hat{L}$. Let $\hat{G}=A^{\phi}*_{C^{\phi}}(B^{\phi})^t$; clearly $\hat{G}\leq\hat{L}$. We claim that $G\cong\hat{G}$. 

This can be showed using normal forms. In particular, the maps $\hat{\phi|_A}:A\rightarrow\hat{G}$ and $\hat{\phi|_B}:B\rightarrow\hat{G}$ defined by composing the restrictions of $\phi$ to $A$ and $B$ respectively with the natural embedding of $A^{\phi}$ and $B^{\phi}$ into $\hat{G}$ (and in the case of $\hat{\phi|_B}:B\rightarrow\hat{G}$, composing with conjugation by the stable letter $t$ in between), are injective (only the conjugation need be checked, and this follows since $t$ is stable letter for $\hat{L}$) by Theorem 1.6 of \cite{SW79}. So define $\hat{\phi}:G\rightarrow\hat{G}$ by $a_1b_1\ldots a_nb_n\mapsto a_1^{\phi}t^{-1}b_1^{\phi}t\ldots a_n^{\phi}t^{-1}b_n^{\phi}t$ for each word in $G$ in normal form ($a_i,b_i\notin C$). Since if $b_i\in C$ then $t^{-1}b_i^{\phi}t=b_i^{\phi}$ and $a_i\notin C$ so $a_i^{\phi}\notin\hat{C}$, we have $a_1^{\phi}t^{-1}b_1^{\phi}t\ldots a_n^{\phi}t^{-1}b_n^{\phi}t$ in normal form. Since every element of $\hat{G}$ has normal form of $a_1^{\phi}t^{-1}b_1^{\phi}t\ldots a_n^{\phi}t^{-1}b_n^{\phi}t=\phi{a_1b_1\ldots a_nb_n}$, $\phi$ is an isomorphism. Then $G$ embeds into $\hat{L}$ and so by Proposition \ref{prop:iter-ext}, $G$ is a $\Gamma$-limit group.
The case for $G=A*_C$ is similar, since $\hat{G}=A^{\phi}*_{C^\phi}$ (and using the fact that if images of $C^\phi$ don't coincide then they must be conjugate by commutative transitivity) again embeds into $\hat{L}$.
\end{proof}

To prove the converse, we use the Bass-Serre tree.

\begin{prop}\label{construct:IGD}Every $\Gamma$-limit group can be constructed by iterated generalized doubles over $\Gamma$.
\end{prop}

\begin{figure}[ht!]
\labellist
\small\hair 2pt
\pinlabel $\textcolor{red}{L}$ at 114 36
\pinlabel $\textcolor{blue}{C}$ at 143 76
\pinlabel $\textcolor{red}{tL}$ at 174 113
\pinlabel $\textcolor{blue}{g_1C}$ at 124 96
\pinlabel $\textcolor{red}{g_1tL}$ at 133 138
\pinlabel $\textcolor{blue}{g_nC}$ at 75 80
\pinlabel $\textcolor{red}{g_ntL}$ at 44 106
\pinlabel $\textcolor{blue}{tC}$ at 208 102
\pinlabel $\textcolor{red}{t^2L}$ at 243 84
\pinlabel $\textcolor{blue}{t^2C}$ at 257 58
\pinlabel $\textcolor{red}{t^3L}$ at 265 30
\pinlabel $\textcolor{red}{t^2g_1tL}$ at 311 56
\pinlabel $\textcolor{blue}{t^2g_1C}$ at 288 67
\pinlabel $\textcolor{red}{tg_1tL}$ at 248 145
\pinlabel $\textcolor{blue}{tg_1C}$ at 219 133
\pinlabel $\textcolor{red}{tg_1t^2L}$ at 309 124
\pinlabel $\textcolor{blue}{tg_1tC}$ at 280 136
\pinlabel $\textcolor{red}{tg_1tg_2tL}$ at 293 186
\pinlabel $\textcolor{blue}{tg_1tg_2C}$ at 302 163
\pinlabel $\textcolor{red}{g_1t^2L}$ at 118 240
\pinlabel $\textcolor{red}{g_1t^3L}$ at 194 236
\pinlabel $\textcolor{red}{g_1tg_2tL}$ at 54 208
\pinlabel $\textcolor{red}{g_1t^2g_2tL}$ at 206 266
\pinlabel $\textcolor{blue}{g_1tC}$ at 136 188
\pinlabel $\textcolor{blue}{g_1tg_2C}$ at 74 178
\pinlabel $\textcolor{blue}{g_1t^2C}$ at 159 220
\pinlabel $\textcolor{blue}{g_1t^2g_2C}$ at 145 260

\endlabellist
\centering
\includegraphics[scale=1]{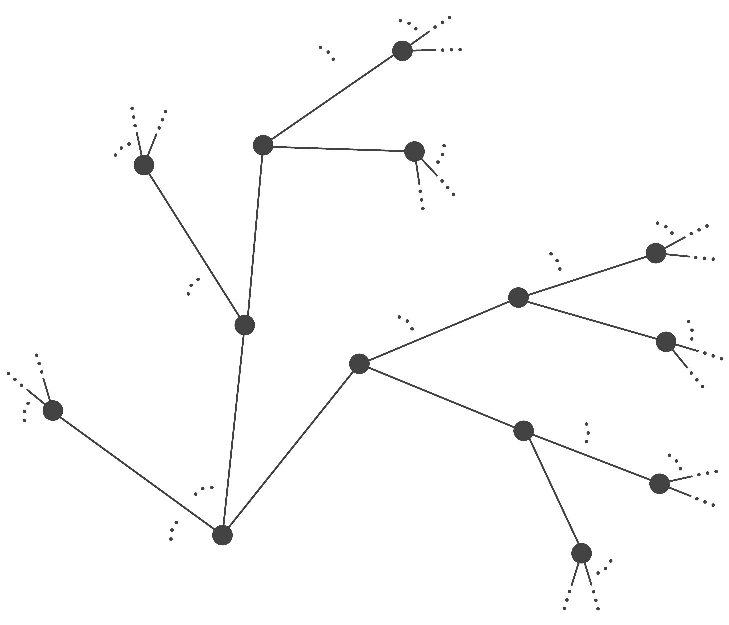}
\caption{Bass-Serre tree}
\label{fig:BStree}

\end{figure}
\begin{proof}
Given a freely indecomposable $\Gamma$-limit group $H$, by Proposition \ref{prop:iter-ext}, it may be embedded into an iterated extension of centralizers $G$ over $\Gamma$. $G$ is an iterated generalized double over $\Gamma$. We claim that every subgroup of $G$ is also an iterated generalized double.

Let $H$ be a subgroup of $G=<L,t|[C_L(U),t]=1>$ for some subset $U$ of an iterated extension of centralizers $L$. Let $C=C_L(U)$. So $G$ is an HNN extension $G=L*_C$. Since $\mathcal{IGD}-\Gamma$ is closed under free product by definition, assume $H$ is freely indecomposable. We will proceed by induction on the number of extensions of centralizers from which $G$ is formed, with the base case being $L=\Gamma$.

By Proposition \ref{prop:BSsubgrp}, $H$ acts (without inversions) on a tree with vertices given by the cosets $gL$ (stabilizers are the groups $L^g\cap H$), and edges $gC$ (corresponding to groups $C^g\cap H$), for each coset representative $g\in G$. In particular the initial vertex of the edge $gC$ is $gL$ and terminal vertex of $gC$ is $gtL$.

Every coset $gL$ has a representative of the form $t^{r_0}g_1t^{r_1}\ldots g_nt^{r_n}$ where the $g_i\in L$ are representatives for cosets of $C$ in $L$, $g_n\neq 1$ and the $r_i$ are non-zero except possibly $r_0$ and $r_n$. Now for each edge $t^{r_0}g_1t^{r_1}\ldots g_nt^{r_n}C$, the terminal vertex is $t^{r_0}g_1t^{r_1}\ldots g_nt^{r_n+1}L$ and the initial vertex is $t^{r_0}g_1t^{r_1}\ldots g_nt^{r_n}L$ if $r_n\neq 0$, and $t^{r_0}g_1t^{r_1}\ldots g_{n-1}t^{r_{n-1}}L$ if $r_n=0$. See Figure \ref{fig:BStree} (vertex cosets are in red, edge cosets in blue). We want to show that the corresponding edge group in the quotient graph of groups is maximal abelian in each of the corresponding vertex groups. Now $Stab(t^{r_0}g_1t^{r_1}\ldots g_nt^{r_n}L)=H\cap L^{t^{-r_n}g_n^{-1}\ldots g_1^{-1}t^{-r_0}}$ and $Stab(t^{r_0}g_1t^{r_1}\ldots g_nt^{r_n}C)=H\cap C^{t^{-r_n}g_n^{-1}\ldots g_1^{-1}t^{-r_0}}$.

Clearly $C^{t^{-r_n}g_n^{-1}\ldots g_1^{-1}t^{-r_0}}\leq L^{t^{-r_n}g_n^{-1}\ldots g_1^{-1}t^{-r_0}}$, and since $C^t=C$ we have $C^{t^{-r_n}g_n^{-1}\ldots g_1^{-1}t^{-r_0}}=C^{t^{-r_n-1}g_n^{-1}\ldots g_1^{-1}t^{-r_0}}\leq L^{t^{-r_n-1}g_n^{-1}\ldots g_1^{-1}t^{-r_0}}$.
\\Finally, $C^{g_n^{-1}t^{-r_{n-1}}\ldots g_1^{-1}t^{-r_0}}\leq L^{g_n^{-1}t^{-r_{n-1}}\ldots g_1^{-1}t^{-r_0}}= L^{t^{-r_{n-1}}\ldots g_1^{-1}t^{-r_0}}$.

A conjugate of a centralizer is again a centralizer in the corresponding conjugate of the ambient group (this can be easily checked directly), and centralizers are maximal abelian. Intersecting a maximal abelian subgroup and it's ambient group with $H$ preserves maximality of the abelian subgroup (again, this can be easily checked directly). So there's a splitting of $H$ with underlying graph $D$ having (non-trivial) abelian edge groups $H_{e_1},\dots,H_{e_m}$ each maximal in adjacent vertex groups.

First, assume $D$ is a tree with vertex groups $H_1=L^{g_1}\cap H,\dots,H_{m+1}=L^{g_{m+1}}\cap H$, for $g_1,\dots,g_{m+1}\in G$ and edge groups $C_1,\dots,C_m$, $C_i\leq C$, with $C_i$ mapping into an adjacent vertex group $H_j$ by $C_i^{g_j}$. Assume $H_1$ corresponds to a leaf of $D$, and that $H_1,\dots,H_m$ are ordered so that each $H_i$ is a leaf of the subtree of $D$ obtained by removing $H_1,\dots,H_{i-1}$, for $i=2,\dots,m-1$.

For $i=0,\dots,m$, let $\hat H_i=\langle H_{i+1}^{g_{i+1}^{-1}},\dots,H_m^{g_m^{-1}}\rangle\leq L$ be the subgroup of $L$ generated by $H_{i+1}^{g_{i+1}^{-1}},\dots,H_m^{g_m^{-1}}$. For new stable letters $t_1,\dots,t_m$, let \[G_i=\langle L,t_1,\dots,t_i|[C,t_1]=\dots=[C,t_i]=[t_j,t_{j'}]=1,1\leq j,j'\leq i\rangle\]
and now consider $H_1*_{C_1} \hat H_1^{t_1}$ with $C_1\mapsto C_1^{g_1}\leq H_1$ and $C_1\mapsto C_1=C_1^{t_1}\leq H_2^{g_2^{-1}t_1}\leq \hat H_1^{t_1}\leq G_1$. 

We claim that $H_1*_{C_1} \hat H_1^{t_1}$ is a generalized double over $\hat H_0\leq L$. From above, $C_1^{g_1}$ and $C_1^{g_2}$ are maximal abelian in $H_1$ and $H_2$ respectively, so $C_1=C_1^{g_2g_2^{-1} t_1}$ is maximal abelian in $H_2^{g_2^{-1} t_1}\leq \hat H_1^{t_1}$ and so maximal in $\hat H_1^{t_1}$. %since otherwise it would be contained in a strictly larger subgroup of 
Conjugation of $H_1$ by $g_1^{-1}$ and of $\hat H_1$ by $t_1^{-1}$ map into $\hat H_0$ and agree on the respective images of $C_1$, so determine (up to isomorphism) $\phi_1: H_1*_{C_1} \hat H_1^{t_1}\to \hat H_0$. Clearly $\phi_1$ is an epimorphism and is injective in restriction to either $H_1$ or $\hat H_1$. 

Similarly, for $i=2,\dots m-1$ \[\left(H_1*_{C_1}H_2^{t_1}*_{C_2}*\dots*_{C_{i-1}}H_i^{t_1\cdots t_{i-1}}\right)*_{C_i} \hat H_i^{t_1\cdots t_i}\] is a generalized double over \[\left(H_1*_{C_1}H_2^{t_1}*_{C_2}*\dots*_{C_{i-2}}H_{i-1}^{t_1\cdots t_{i-2}}\right)*_{C_{i-1}} \hat H_{i-1}^{t_1\cdots t_{i-1}}\]
with the epimorphism determined by the identity on $H_1,H_2^{t_1},\dots,H_{i-1}^{t_1\dots t_{i-2}}$, conjugation of $H_i^{t_1\dots t_{i-1}}$ by $\left(g_i t_1\dots t_{i-1}\right)^{-1}$ and of $\hat H_i^{t_1\dots t_i}$ by $\left(t_1\dots t_i\right)^{-1}$.

Now $\hat H_{m-1}=H_m^{g_m^{-1}}$ and $H\cong H_1*_{C_1} H_2^{t_1}*_{C_2}\cdots*_{C_m}  H_{m}^{g_m^{-1}t_1\cdots t_m}$ (transforming $D$ by sliding and conjugation), and so $H$ is formed by a sequence of generalized double constructions, starting with a subgroup of $L$, so $H$ is in $\mathcal{IGD}-L$.

Let $\bar H$ be the image of $H$ in $L$ when $t$ is mapped to the identity. In the general case we will use induction on the number of edges in $D$. If there are no edges corresponding to HNN extensions, then the statement has been already proved. Suppose 
there is a loop $v_1-v_2-\ldots v_m-v_1$ in the graph $D$  with corresponding  vertex groups 
$$H_1 *_{(C_1^s=C_2)} H_2  *_{C_3}\ldots H_m *_{C_{m+1}} H_1,$$ where $s$ is a stable letter corresponding to the edge  $e=(v_1, v_2)$.    Let $\bar H_i$ be the image of $H_i$ in $\bar H$ and $\hat H$ be the image of the fundamental group of the graph of groups obtained from $D$ by  removing $e$. Consider an extension of centralizer $L_1=\langle \bar H, t_1|[C_1,t_1]=1\rangle$ and let $L_2$ be the subgroup of $L_1$ obtained as a generalized double over $\bar H$ corresponding to the HNN-extension $\langle \hat H, t_1s| C_1^{t_1s}=C_2\rangle.$

We can now consider $H$ as the fundamental group of  the new graph $\bar D$ obtained from $D$ by collapsing  the edge $e$.  We define  the epimorphism  $\phi (H)\rightarrow L_2$ as follows: $\phi (H_1)=\bar H_1^{ts}$, $\phi (H_i)=\bar H_i$ for $i>1$,  and stable letters are mapped naturally (the stable letter $t_1s$ will now correspond to the edge $(v_n,v_1)$). Then  the group $H$ is the fundamental group of the graph of groups $\bar D$ with less edges, it admits the epimorphism $\phi$ to $L_2$ that maps the vertex groups including $H_1*_{C_1}H_2$ monomorphically. 
Therefore, by induction on the number of edges in $D$, $H$ is obtained by a sequence of generalized doubles from $L_2$, which is a generalized double over $\bar H$. Therefore $H\in \mathcal{IGD}-L$.  We still have to consider the case when $D$ has only one vertex, that is $H$ is a multiple HNN extension $H=\langle H_1,t_1,\ldots ,t_k| [C_i,t_i]=1\rangle$  that is in $\mathcal{IGD}-L$ by Lemma \ref{extCent}. In this case we take $\langle H_1,t_1| [C_1,t_1]=1\rangle$ as a new vertex group and  $\bar D$ is obtained from $D$ by collapsing $e_1$.

The case $L=\Gamma$ puts subgroups of a single extension of centralizer of $\Gamma$ in the class $\mathcal{IGD}-\Gamma$. Then considering an extension of centralizer $G=\langle L, t|[C,t]=1\rangle$ over an iterated extension of centralizers $L$ (which starts with $\Gamma$), where all subgroups of $L$ are in $\mathcal{IGD}-\Gamma$, we've showed that subgroups of $G$ are in $\mathcal{IGD}-L$, in which case they are then in $\mathcal{IGD}-\Gamma$.

\end{proof}

%Theorem \ref{Gamma-IGD} follows (since free products of $\Gamma$ limit groups are also $\Gamma$-limit groups).

\section{JSJ decompositions}\label{section:canonical-splittings}

In this section we describe properties of JSJ decompositions of groups. We refer to \cite{RS97}, \cite{DS}, \cite{FP06}, and \cite{GL} for further background on other notions of JSJ decompositions for groups.

Recall that an abelian splitting of a group $G$ is an isomorphism to the fundamental group of a graph of groups with all abelian edge groups. It is often convenient to slightly abuse terminology and allow a splitting to refer to the graph of groups itself. An \emph{elementary} splitting is one in which the graph of groups has exactly one edge. A splitting is \emph{reduced} if the image of each edge group is a proper subgroup of the corresponding vertex group. A splitting of a group $G$ is \emph{essential} if it is reduced, all edge groups are abelian and for each edge group $E$, if $\gamma\in G,\gamma^k\in E$ for some $k\neq 0$, then $\gamma\in E$. 

A subgroup $H\leq G$ is said to be \emph{elliptic} with respect to a given splitting, if it is conjugate to a subgroup of some vertex group of the splitting. If $H$ is not elliptic, then it is called \emph{hyperbolic} with respect to the splitting. 

There are certain elementary transformations of graphs of groups which preserve the fundamental group.

An \emph{unfolding} of an elementary splitting $G\cong A*_CB$ is another splitting $G\cong A*_{C_1}B_1$ where $C_1$ is a proper subgroup of $C$ and $B=C*_{C_1}B_1$. An \emph{unfolding} of an elementary splitting $G\cong A*_C$ is another splitting $G\cong A_1*_{C_1}$ where the image of $C_1$ in $A_1$ is a proper subgroup of $C$, and $A=A_1*_{C_1}(tCt^{-1})$ where $t$ is the stable letter of $A*_C$. We say a splitting is \emph{unfolded} if there is no unfolding of any induced elementary splitting.

Vertex groups may be divided into three classes. 
\begin{enumerate}
\item \emph{Abelian} vertex groups

\item \emph{Quadratically hanging} vertex groups. A vertex group $G_v$ is quadratically hanging (or QH), if it admits either of the presentations (i.e. it is the fundamental group of a surface with finitely many punctures):
\begin{enumerate}[(i)]
\item $\displaystyle\langle a_1,\ldots,a_g,b_1,\ldots,b_g|\prod_{i=1}^{g}[a_i,b_i]\prod_{j=1}^{m}p_j\rangle$ where  $g\geq0,m\geq 1$, and if $g=0$ then $m\geq 4$
\item $\displaystyle\langle a_1,\ldots,a_g,p_1,\ldots,p_m|\prod_{i=1}^{g}a_i^2\prod_{j=1}^{m}p_j\rangle$ where $g\geq1,m\geq1$

\end{enumerate}
and furthermore:
\begin{itemize}
\item for each edge group $G_e$ with $\partial_0(e)=v$, $\alpha_e(G_e)$ is conjugate to $\langle p_j\rangle$ in $G_v$ for some $j=1,\ldots,m$
\item and for each $j=1,\ldots,m$ there is some edge $e_j$ with $\partial_0(e_j)=v$ and $\alpha_{e_j}(G_{e_j})$ conjugate to $\langle p_j\rangle$ in $G_v$.
\end{itemize}
\item Vertex groups which are non-abelian and non-QH are called \emph{rigid} vertex groups.
\end{enumerate}

Any subgroup $H\leq G$ which is a QH vertex group in some splitting of $G$ is called a \emph{QH subgroup} of $G$. A QH subgroup $Q\leq G$ is a \emph{maximal quadratically hanging subgroup} (or MQH subgroup), if given any elementary splitting of $G$ with edge group $G_e$, either $Q$ is elliptic in that splitting, or $G_e$ can be conjugated into $Q$ and that elementary splitting is induced by splitting $Q$ along the image of $G_e$.

There are various definitions of splittings which are called JSJ decompositions (after the topological notion of decomposing 3-manifolds along essential tori due to Jaco, Shalen, and Jacobsen) for various classes of groups. All of them are canonical in that they encode all other splittings (of a certain type) in some sense. We will use some particular qualifications for the canonical decompositions of $\Gamma$-limit groups.

\begin{definition}Given a reduced unfolded  abelian splitting $D$ of a freely indecomposable $\Gamma$-limit group $L$, call $D$ a JSJ decomposition of $L$, if the following properties are satisfied:
\begin{enumerate}[(i)]
\item Every MQH subgroup of $L$ is conjugate to a vertex group of the $D$, (so every QH subgroup of $L$ can be conjugated into a vertex group of the JSJ), and every vertex group of the $D$ which is not conjugate to a MQH subgroup of L is elliptic in any abelian subgroup of $L$.
\item Any elementary abelian splitting of $L$ which is hyperbolic in another elementary abelian splitting, can be obtained from $D$ by the splitting of an MQH subgroup, which is induced by cutting the corresponding surface along an essential simple closed curve, and collapsing all other edges.
\item Any elementary abelian splitting of $L$ which is elliptic with respect to every other elementary elementary abelian splitting of $L$, can be obtained from $D$ by a sequence of collapsings, foldings, and conjugations.
\item Any two reduced unfolded abelian splittings of $L$ satisfying the above three properties can be obtained from one another by a sequence of slidings, conjugations, and modification of boundary monomorphisms by conjugations.
\item All non-cyclic abelian subgroups of $L$ are elliptic in $D$.
\end{enumerate}
\end{definition}

In \cite{selaHyp} (Theorem 1.10), every freely indecomposable $\Gamma$-limit group is shown to have such a JSJ decomposition.

\section{Hom-diagrams and $\Gamma$-NTQ groups}\label{section:homoms}
We present here some key constructions from \cite{KMlyndon} and \cite{embedExt}, along with generalizations of some  from \cite{elemFree} to the torsion-free  hyperbolic group case. Building on those foundations, for a finite system of equations over $\Gamma$ we construct a particular solution tree, while recalling the existence of  canonical Hom-diagrams. Then we establish some necessary language and tools to relate our tree to canonical Hom-diagrams: corrective extensions of $\Gamma$-NTQ groups, certain ``generic families" of solutions, and lifting of solutions using the Implicit Function Theorem.

Unless otherwise stated, for systems with coefficients in $\Gamma$, all homomorphisms in Hom-diagrams for the system are assumed to be $\Gamma$-homomorphisms.
\subsection{Canonical automorphisms}
Consider an elementary abelian splitting of a group $G$. If we have $G=A*_{C}B$, for $c\in C$ define an automorphism $\phi_c:G\rightarrow G$ such that $\phi_c(a)=a$ for $a\in A$ and $\phi_c(b)=b^c=c^{-1}bc$ for $b\in B$. If instead we have $G=A*_c$ then for $c\in C$ we define $\phi_c:G\rightarrow G$ such that $\phi_c(a)=a$ for $a\in A$ and $\phi_c(t)=ct$.

Call the $\phi_c$ a \emph{Dehn twist} obtained from the corresponding elementary abelian splitting of $G$. If $G$ is an $\Gamma$-group, where $\Gamma$ is a subgroup of one of the factors $A$ or $B$, then Dehn twists that fix elements of the  group $\Gamma\leq A$ are called \emph{canonical Dehn twists}. Similarly, one can define canonical Dehn twists with respect to an arbitrary fixed subgroup $K$ of $G$. Let $\mathcal{D}(G)$ [resp. $\mathcal{D}_{\Gamma}(G)$]  be the set of all abelian splittings of $G$ [resp. the set of all abelian splittings such that $\Gamma\leq G$ is elliptic].
\begin{definition}
Let $D\in \mathcal{D}(G)$ [$D\in \mathcal{D}_{\Gamma}(G)$] be an abelian splitting of a group $G$ and $G_v$ be either a QH or an abelian vertex of $D$. Then an automorphism $\psi\in Aut(G)$ is called a \emph{canonical automorphism} corresponding to the vertex $G_v$ if $\psi$ satisfies the following conditions:
\begin{enumerate}[(i)]
\item $\psi$ fixes (up to conjugation) element-wise all vertex group in $D$, other than $G_v$. If $\Gamma\leq G_v$ then $\psi$ also fixes each element of $\Gamma$. Note that $\psi$ also fixes (up to conjugation) all the edge groups.
\item If $G_v$ is a QH vertex in $D$, then $\psi$ is a Dehn twist (canonical Dehn twist) corresponding to some essential $\mathbb{Z}$-splitting of $G$ along a cyclic subgroup of $G_v$.
\item If $G_v$ is an abelian subgroup then $\psi$ acts as an automorphism on $G_v$ which fixes all the edge subgroups of $G_v$.
\end{enumerate}
\end{definition}
\begin{definition}
Let $e$ be an edge in an abelian splitting $D$ of a group $G$. Let $\psi\in Aut(G)$. Call $\psi $ a \emph{canonical automorphism corresponding to the edge $e$} if $\psi$ is a Dehn twist of $G$ with respect to the elementary splitting of $G$ along the edge $e$, induced from the splitting $D$. If $D\in \mathcal{D}_{\Gamma}(G)$, then $\psi$ must fix $\Gamma$ element wise.
\end{definition}
\begin{definition}The group of canonical automorphisms of a closed surface group $G$ with respect to a trivial splitting $D$ (i.e. $G$ is the only vertex group) is $A_D(G)=MCG(\Sigma)$ (the mapping class group of the surface).

Otherwise, the group of canonical automorphisms of a freely indecomposable group $G$ with respect to an abelian splitting $D$ is the subgroup $A_D(G)\leq Aut(G)$, generated by all canonical automorphisms of $G$ corresponding to all edges, all $QH$ vertices, and all abelian vertices of $D$. If $G$ is not a $\Gamma-group$ then include conjugation as well. The group of canonical automorphisms of a free product is the direct product of the canonical automorphism groups of factors.
\end{definition}

\subsection{Solution tree $\mathcal{T}(S,\Gamma)$}\label{section:reworking}
 
 \begin{definition} [{\bf fundamental sequence}] \label{non-can} Let $\Gamma _{R(S)}$ be a coordinate group over $\Gamma$.
Consider the following diagram \begin{equation}\label{canhom} \Gamma_{R(S)}\rightarrow _{\pi_0} G_1\rightarrow _{\pi_1}  G_2\rightarrow\ldots \rightarrow _{\pi_{n-2}}  G_{n-1}\rightarrow _{\pi_{n-1}}  \bar G_n\rightarrow _{\tau}G_{n}\leq F(Y)*\Gamma ,\end{equation}
where $Y=\{y_1,\ldots , y_m, \ldots, y_{m+k}\},$ $\bar G_n=F(y_1,\ldots ,y_m)\ast \Gamma \ast H_1\ast\ldots \ast H_k,$
\newline $G_{n}= F(y_1,\ldots ,y_m)\ast \Gamma \ast \Gamma ^{y_{m+1}}\ast\ldots \ast \Gamma ^{y_{m+k}} ,$
with the following properties:
\begin{enumerate}
\item $H_1,\ldots , H_k$ are  groups isomorphic to subgroups of $\Gamma$,  and $G_1,\ldots ,G_n$ are $\Gamma$-limit groups.
 
\item $\pi _0$ is a homomorphism, $\pi_i$ is a fixed proper epimorphism for $0<i<n-1$,  $\pi _{n-1}$ is an epimorphism, but  may not be proper, and $\tau$ is a fixed homomorphism, which is a monomorphism when restricted to $H_j$, and there is a new variable $y_{m+j}$ with $\tau (H_j)\leq \Gamma ^{y_{m+j}} $ for each $1\leq j\leq k.$
\end{enumerate}
 
 A {\em fundamental sequence}  of homomorphisms from $\Gamma _{R(S)}$ to $F(Y)\ast\Gamma$ corresponding to this diagram is the set of all homomorphisms that are compositions $$\pi_0\sigma_1\pi _1\ldots\sigma _{n-1}\pi _{n-1}\tau ,$$ where $\sigma _i$ is a canonical automorphism of $G_i$ corresponding to a Grushko decomposition of $G_i$ followed by some abelian decompositions of the freely indecomposable factors where  all non-cyclic abelian subgroups are elliptic. 
 
 A {\em fundamental sequence}  of homomorphisms from $\Gamma _{R(S)}$ to $\Gamma$  consists of all homomorphisms  from the fundamental sequence from $\Gamma _{R(S)}$ to $F(Y)\ast\Gamma$ post-composed with arbitrary homomorphisms $\delta$ that map $F(Y)$ into $\Gamma$.

\end{definition}
\begin{definition}  [{\bf strict fundamental sequence}]
A fundamental sequence defined above is called {\em strict}  if it has the following properties: \begin{enumerate} 
\item The image of each non-abelian vertex group of $G_i$ under $\pi _i$ is non-abelian.
\item For each $1\leq i< n$, $\pi_i$ is injective on rigid subgroups, edge groups, and subgroups generated by the images of edge groups in abelian vertex groups in $G_{i-1}$.
\item For each $1\leq i <n $, if $R$ is a rigid subgroup in $G_{i}$ and $\{A_j\}$, $1\leq j\leq m$, the abelian vertex groups in $G_{i}$ connected to $R$ by edge groups $E_j$ with the maps $\eta_j:E_j\rightarrow A_j$, then $\pi_i$ is injective on the subgroup $\langle R,\eta_1(E_1),\ldots,\eta_m(E_m)\rangle$ which we will call the {\em envelope of $R$}.
\item The images of different factors in the free decomposition of $G_i$ under $\pi _i$  are different factors in the free decomposition of $G_{i+1}$. 
\end{enumerate}
\end{definition}

All homomorphisms from $\Gamma _{R(S)}$ to $\Gamma$ belong to a finite number of strict fundamental sequences to $\Gamma$.  The union of their diagrams is called a Hom-diagram or Makanin-Razborov diagram. To each  strict fundamental sequence one  assigns a $\Gamma$-NTQ  group similarly to how it is done in \cite{elemFree} or \cite{selaHyp}.  
 \begin{lemma}(\cite{embedExt}, Lemma 3.15)\label{Lem:EmbeddingIntoCoordinateGroups} Let $\Gamma$ be a torsion-free hyperbolic group.
There is an algorithm that, given a finitely presented group $G=\langle Z\gst S\rangle$,
produces
\begin{romanenumerate}
\item finitely many fully residually $\Gamma$ groups $N_{1},\ldots,N_{m}$, and %$\Gamma$-NTQ systems $T_{1}(X_{1},A),\ldots, T_{m}(X_{m},A)$, and
\item homomorphisms $\alpha_{i}: G\rightarrow N_{i},$ $i=1,\ldots ,m.$
\end{romanenumerate}
such that for every homomorphism $\psi: G\rightarrow \Gamma$ there exists a homomorphism
$\hat{\phi}: N_{i}\rightarrow \Gamma$ such that $\psi=\alpha_{i}\hat{\phi}$.  Further, each group $N_{i}$ has the form $N_{i}=\Gamma_{R(T_{i})}$ where
$T_{i}$ is a $\Gamma$-NTQ system with the base group $\Gamma * F(Y)$ (or, more, precisely, $F(y_1,\ldots ,y_m)\ast \Gamma \ast \Gamma ^{y_{m+1}}\ast\ldots \ast \Gamma ^{y_{m+k}}\leq  \Gamma * F(Y)) ,$ where $Y$ is some finite set of variables, and the homomorphism $\hat{\phi}$ is a $\Gamma$-homomorphism that belongs to the strict fundamental sequence associated with $N_{i}$.

The lemma also holds in the category of $\Gamma$-groups, i.e. for systems of equations with constants over $\Gamma$.
\end{lemma} 
Noe we re-group into levels the equations of NTQ systems from Lemma  \ref{Lem:EmbeddingIntoCoordinateGroups}. For each  $\Gamma$-NTQ group $N_i=N=\Gamma _{R(T)}$   consider the partition (from the NTQ structure) of $T=1$ into subsystems $T_j=1$, $j=1,\ldots , n-1,$ as in Definition \ref{NTQ}, and let $L_j$ be the $\Gamma$-NTQ group of level $j$ that is the coordinate group of the union of $T_j=1,\ldots , T_{n-1}=1$ over the base group  $L_n=F(y_1,\ldots ,y_m)\ast\Gamma \ast \Gamma ^{y_{m+1}}\ldots  \ast \Gamma ^{y_{m+k}}.$ Here $N=L_1$ is the group of the top level.
We have
$N=L_1>L_2>\ldots >L_n$.  To each $L_j$ there is an associated splitting  $D_j$ that is a free decomposition followed by  abelian splittings of factors. In these abelian splittings QH subgroups correspond to quadratic equations in $T_j=1$, abelian vertex groups correspond to equations of type (II) and rigid vertex groups are non-cyclic factors in the free decomposition of  $L_{j+1}$. So, given a fixed NTQ structure for a group and taking canonical automorphisms  with respect to splittings associated to each level we obtain a strict fundamental sequence.

Applying this to the groups constructed in Lemma  \ref{Lem:EmbeddingIntoCoordinateGroups} we  define 
the  tree ${\mathcal T} (S, \Gamma)$ (which, basically, comes from the paper \cite{embedExt}).

\begin{definition} [{\bf tree ${\mathcal T} (S, \Gamma)$}] \label{KMac}  A {\em Solution tree}  for a system $S=1$ over $\Gamma$ is a directed rooted tree ${\mathcal T} (S, \Gamma)$, with vertices labeled by groups and edges labeled by homomorphisms. Each branch $b$, i.e. a path from the root vertex to a leaf, has the following form  $$\Gamma_{R(S)}\rightarrow _{\phi_b} N_b=L_1\rightarrow _{\psi_1}  L_2\rightarrow\ldots \rightarrow _{\psi_{n-1}}  L_n=F(y_1,\ldots ,y_m)\ast\Gamma \ast \Gamma ^{y_{m+1}}\ldots  \ast \Gamma ^{y_{m+k}}, $$  where $L_1,\ldots ,L_n$ are $\Gamma$-NTQ groups with the  base group $L_n$; $\phi _b$ is a fixed homomorphism; and  $\psi_j$ are fixed proper epimorphisms which are retractions onto $L_{j+1}$ for $0<j\leq n-1$. 

There is a {\em completed fundamental sequence} to $F(Y)\ast\Gamma$  and a completed fundamental sequence to $\Gamma$ assigned to each branch. The homomorphisms in the  fundamental sequence  to $F(Y)\ast\Gamma$ are compositions $\phi _b\sigma_1\psi _1\ldots\sigma _{n-1}\psi _{n-1}$ where $\sigma _j$ is a canonical automorphism with respect to the splitting associated  with  $L_j$ that is coming from the NTQ structure.  The homomorphisms in the  fundamental sequence to $\Gamma$ are obtained by post-composing  the homomorphisms in the fundamental sequence to $F(Y)\ast\Gamma$ with arbitrary  homomorphism  $\delta$  that maps $F(Y)$ into $\Gamma$.  

Every homomorphism from $\Gamma _{R(S)}$ to $\Gamma$ factors through one of the completed fundamental sequences to $\Gamma$ corresponding to the branches of ${\mathcal T} (S, \Gamma)$.
\end{definition}

Notice, that the base group $L_n$ for each branch is hyperbolic and every group $L_i$ is toral relatively hyperbolic.  A completed fundamental sequence is always strict.
\begin{lemma}\label{le:spleff}
The tree ${\mathcal T} (S, \Gamma)$ can be effectively constructed. In addition  one can 
assume that for each branch $b$, $\phi _b(\Gamma _{R(S)})$ is not conjugate into a factor in a non-trivial free decomposition of $N_b$ or into $\psi _1(N_b)$. 
We can further assume that there is no vertex group in the splitting corresponding to the top level of $N_b$ such that $\phi _b(\Gamma _{R(S)})$ is conjugate in the fundamental group of the graph of groups obtained by removing this vertex.  And for each level $L_j$  of $N_b$ we can assume the same about  $\psi _{j-1}\ldots \psi _1\phi _b(\Gamma _{R(S)})$ in $L_j$. \end{lemma}

\begin{proof} By Lemma \ref{Lem:EmbeddingIntoCoordinateGroups} the tree can be algorithmically constructed. By  \cite{GW}, Theorem  7.2 there is an algorithm to decide if  $\phi _b(\Gamma _{R(S)})$ is  conjugate into the $\Gamma$-NTQ group $L_2$ corresponding to the second level of the $\Gamma$-NTQ group $N_b$. If it is, we can just replace $N_b$ by $L_2$. For each level $j$ beginning with the top level we can similarly check  if there is a vertex group in the splitting corresponding to $L_j$  such that $\psi _{j-1}\ldots \psi _1\phi _b(\Gamma _{R(S)})$  is conjugate in the fundamental group of the graph of groups obtained by removing this vertex. If it is, we remove this vertex group from the graph of groups representing $L_j$.

\end{proof} 

We will need the following  lemma from \cite{effectiveJSJ}.

\begin{lemma}\label{le:qh} \label{le:spl} (Lemma 2.13 in \cite{effectiveJSJ}) 
Let $\Delta = \Delta (X) $ be a
cyclic {\rm[}abelian{\rm]} splitting of the group $G$, and $Q$ a QH-subgroup
in $\Delta$ associated with a vertex $v \in X$ with outgoing edges
$e_1, \ldots, e_m$ corresponding to boundary subgroups $\langle p_1\rangle ,\ldots  \langle p_m\rangle .$ Denote by $Y_1, \ldots, Y_k$ the connected
components of the graph $X \smallsetminus \{v, e_1, \ldots, e_m\}$
and by $P_1, \ldots, P_k$ - the fundamental groups of the graphs
of groups induced from $\Delta$ on $Y_1, \ldots, Y_k$. If $H$ is a
finitely generated non-cyclic subgroup of $G$ then one of the
following conditions holds:
 \be
 \item  $H$ is a nontrivial  free product;
 \item  $H \leq P_i^g$ for some $g \in G$ and $1 \leq i \leq k$;
 \item  $H$ is freely indecomposable, and for some $g\in G$
 the subgroup $H\cap Q^g$ has finite index in $Q^g$.
In this event $H \cap Q^g$ is a QH-vertex group in $H$. \ee

 If $H_Q=H\cap Q$ is non-trivial and has infinite index in $Q$,
 then $H_Q$ is a free product of some conjugates of $\langle p_1^{t
 _1}\rangle,\ldots ,\langle p_m^{t _m}\rangle,$ $t_i\in\mathbb Z$, and a free group $F_1$ {\rm(}maybe trivial{\rm)} which
 does not intersect any conjugate of $\langle p_i\rangle $ for
 $i=1,\ldots ,m.$

\end{lemma}

\subsection{Canonical $\Gamma$-NTQ groups}\label{Hom} 
In  the previous section we described the tree ${\mathcal T} (S, \Gamma)$ of  completed fundamental sequences (or a $Hom$-diagram) encoding all solutions of a finite system $S(Z,A)=1$ of equations over $\Gamma$. A {\em canonical $Hom$-diagram} is a tree of strict fundamental sequences such that the group of canonical automorphisms of each $G_i$ corresponds to the JSJ decomposition of $G_i$ (not to some splitting of $G_i$) and  each group $G_{i+1}$ is one of the maximal $\Gamma$-limit quotients of the 
 shortening quotient of $G_i$ (which is the quotient of $G_i$ over the intersection of the kernels of all minimal homomorphisms into $\Gamma$ with respect to $Aut_D(G_i)$, where $D$ is the JSJ decomposition of $G_i$). A canonical Hom-diagram is not unique. There is some freedom in the constructions, from \cite{selaHyp} and \cite{embedExt}, used to show the existence of such a diagram. While the initial step of forming strict fundamental sequences that discriminate maximal $\Gamma$-limit quotients of $\Gamma _{R(S)}$ is unique, there are choices for selecting their proper quotients. In particular, adding a branch from a canonical Hom-diagram for a proper quotient of a maximal $\Gamma$-limit quotient of $\Gamma_{R(S)}$, to a given canonical Hom-diagram for $\Gamma_{R(S)}$, results in another canonical Hom-diagram of $\Gamma_{R(S)}$.

To each  strict fundamental sequence of any Hom-diagram one can assign a $\Gamma$-NTQ  group and a completed fundamental sequence similarly to how it is done in \cite{elemFree} or \cite{selaHyp}.  
So in this way to  each fundamental sequence (\ref{canhom}) (Definition \ref{non-can}) of a canonical Hom-diagram, we assign a $\Gamma$-NTQ group  and a completed  fundamental sequence. 
The obtained Hom-diagram of completed fundamental sequences is called {\em canonical completed} Hom-diagram. Those $\Gamma$-NTQ groups from this diagram that embed maximal $\Gamma$-limit quotients of $\Gamma _{R(S)}$ are called {\em canonical $\Gamma$-NTQ groups of $\Gamma _{R(S)}$}.

 %The set of all NTQ groups corresponding to a canonical $Hom$-diagram is {\em a full set of canonical $\Gamma$-NTQ groups}. 
 
%\begin{remark} When considering fundamental sequences for NTQ groups (not necessary canonical NTQ groups) we only consider fundamental sequences such that canonical epimorphisms between the levels are retractions onto the lower level. They will be necessarily strict.
%\end{remark} 
 
 An algorithm is given in Section \ref{section:can-construct}, which constructs a set of $\Gamma$-NTQ groups containing canonical $\Gamma$-NTQ groups of a completed canonical   Hom-diagram through which all solutions of $S=1$ factor.  
 \subsection{Implicit Function Theorem}\label{subsection:IFT}

Let $N$ be a $\Gamma$-NTQ group. {\em Corrective extensions} of $N$  are obtained by

(i) Replacing each of the free abelian groups that appear in the Grushko  decompositions on different levels of the $\Gamma$-NTQ group by a free abelian group of the same rank, that contains the original one as a subgroup of finite index.

(ii) Replacing each of the free abelian vertex groups that appear in the abelian decompositions of freely indecomposable factors on different levels of the $\Gamma$-NTQ group by a free abelian group of the same rank, that contains the original one as a subgroup of finite index.

Notice that the last operation  does not make elements that aren't conjugate, conjugate (see \cite{LT2018}, page 195).  We define the abelian size of a $\Gamma$-NTQ group $N$, denoted $ab(N)$, as the sum of the ranks of the abelian vertex groups in decompositions corresponding to different levels of $N$ minus the sum of the ranks of their direct summands containing edge groups as subgroups of finite index. Then $ab(N)$ is the same as $ab(N_{corr})$ for each corrective extension $(N_{corr})$ of $N$.

\begin{definition} \label{cover} A set of corrective extensions of $\Gamma$-NTQ groups in a completed  $Hom$-diagram   for  $\Gamma _{R(S)}$ is called a {\em cover} of the diagram if every solution of $S=1$  factors through one of the corrective extensions % {\ok and for any equation $V(X,Y)=0$ if for any $\bar X$ solving $S(X)=1$ there exists $Y$ solving $V(\bar X,Y)=1$, then there exists a formal solution $Y$ in one of these corrective extensions.}
\end{definition}

Example. Consider the NTQ group $F \ast \langle x,y|[x,y]=1\rangle ,$
 Consider three corrective extensions  
 
 $F \ast \langle x,y,z|[x,z]=1, y=z^2\rangle ,$ $F \ast \langle x,y,z|[y,z]=1, x=z^2\rangle $; 
 
 $F \ast \langle x,y,z|[x,z]=1, [y,x]=1,xy=z^2\rangle .$
 
 They form a cover, but any two of them don't form a cover.
 
For the formulation of the technical version of the implicit function theorem we need the notion of a {\em generic family of solutions} of an NTQ system $W(X,A)=1$ (or test sequence in Sela's terminology). This is, in particular, an infinite discriminating family of homomorphisms from $\Gamma _{R(W)}$ to the
free product $F(y_1,\ldots ,y_m)\ast\Gamma\ast \Gamma ^{y_{m+1}}\ast\ldots \ast \Gamma ^{y_{m+k}}.$ If a generic family is divided into a finite number of subfamilies, then at least one of the subfamilies is also a generic family for  $W(X,A)=1$. We will define such families in the next subsection. In this paper we only need their existence.

\begin{prop}\label{IFT} (Version of the Implicit Function Theorem)
If $\Psi(W)$ is a generic family of solutions in $\Gamma$ for an NTQ system $W(X,A)=1$, then for any system of equations 
$U(X,Y,A)=1$ the following is true: if for any solution $\psi\in\Psi(W)$
there exists a solution of $U(X^{\psi},Y,A)=1$ in $\Gamma$, then there exists a finite number of corrective extensions $(\Gamma_{R(W)})_k$ of
$\Gamma_{R(W)}$ and for each $k$ there is a homomorphism  $f_k:Y\rightarrow (\Gamma_{R(W)})_k$  such that  $U(X,Y^{f_k},A)=1$ in $(\Gamma_{R(W)})_k$.   Any solution of $W(X,A)=1$ that factors through the completed fundamental sequence for $W(X,A)=1$, factors through the fundamental sequence for one of these corrective extensions. \end{prop}

A solution of the system $U(X,Y,A)=1$ in the corrective extension $(\Gamma_{R(W)})_k$ is called a {\em formula solution}.

This theorem is similar to  the Parametrization Theorem, also called the Implicit Function Theorem (\cite{KMimplicit},Theorem 12) for free groups. A similar result is also formulated  in \cite{selaHyp}, Theorem 2.3 for hyperbolic groups, but the formulation in \cite{selaHyp} contains an error (see \cite{elemHyp} for comments and corrections).  Recently the proof of the result appeared in \cite{Heil}.

\subsection{Generic families of solutions}

One can skip this subsection, it is given for the completeness of the presentation.
We will  describe the construction of  \emph{generic families} of solutions of an NTQ system which are used in the proof of the theorem above in \cite{selaHyp} and \cite{elemHyp}.  
Consider a strict fundamental sequence with corresponding NTQ system $S(X,A)=1$ of depth $N$. We construct generic families iteratively for each level $k$ of the system, starting at $k=N$ and decreasing $k$. There is an abelian decomposition of $G_k$ corresponding to the NTQ structure. Let $V_1^{(k)},\ldots,V_{M_k}^{(k)}$ be the vertex groups of this decomposition given some arbitrary order. We construct a generic family for level $k$, denoted $\Psi(k)$, by constructing generic families for each vertex group in order. We denote a generic family for the vertex group $V_i^{(k)}$ by $\Psi(V_i^{(k)})$. If there are no vertex groups, in other words the equation $S_k=1$ is empty ($G_k=G_{k+1}*F(X_k)$) we take $\Psi(k)$ to be a sequence of growing different Merzlyakov's words (defined in \cite{KMimplicit}, Section 4.4). 

\begin{remark}\label{Merz}When using generic families in this paper,  by \cite{selaHyp}, Proposition 2.1,  instead of a family of growing Merzlyakov's words  in $\Gamma$, one can just take new  letters  for the basis of $F(X_k)$.  So instead of a family of homomorphisms in  $\Gamma$, we can consider generic family as a family of solutions into $\Gamma\ast F(Y)$ for some basis $Y$ consisting of new letters. \end{remark}

If $V_r^{(k)}$ is an abelian group then it corresponds to equations of the form $[x_i,x_j]=1$ or $[x_i,u]=1$, $1\leq i,j\leq s$,  where $u\in U$ runs through a generating set of  a centralizer in $G_{k+1}$. A solution $\sigma$ in $G_{k+1}$ to equations of these forms is called \emph{$B$-large} if there is some $b_1,\ldots,b_s$ with each $b_i>B$ such that $\sigma(x_i)=(\sigma(x_1))^{b_1\ldots b_i}$ or $\sigma(x_i)=u^{b_1\ldots b_i}$, for $1\leq i\leq s$ (possibly renaming $x_1$). A generic family of solutions for an abelian subgroup $V_r^{(k)}$ is a family $\Psi(V_r^{(k)})$ such that for each $B_i$ in any increasing sequence of positive integers $\{B_i\}_{i=1}^{\infty}$ there is a solution in $\Psi(V_r^{(k)})$ which is $B_i$-large.

If  $V_r^{(k)}$ is a QH vertex group of this decomposition, let $S$ be the surface associated to  $V_r^{(k)}$. We associate two collections of non-homotopic, non-boundary parallel, simple closed curves $\{b_1,\ldots b_q\}$ and $\{d_1,\ldots d_t\}$. These collections should have the property that $S-\{b_1\cup\cdots\cup b_q\}$ is a disjoint union of 
%three-punctures spheres and one-punctured Mobius bands, 
connected components with Euler characteristic $-1$,
each of the curves $d_i$ intersects at least one of the curves $b_j$ non-trivially, and their union fills the surface $S$ (meaning the collection $\{b_1,\ldots b_q, d_1,\ldots ,d_t\}$ have minimal number of  intersections and $S-\{b_1,\cup\cdots\cup b_q\cup d_1\cup\cdots\cup d_t\}$ is a disjoint union of connected components  in $S$, where each connected component is either homeomorphic  to a disk or to an annulus.  If a component is homeomorphic to an annulus, then one of its boundary components is a  boundary component of the surface $S$.).

Let $\beta _1,\ldots ,\beta_q$ be automorphisms of $V_r^{(k)}$ that correspond to Dehn twists along $b_1,\ldots b_q$ , and $\delta _1,\ldots ,\delta_t$ be automorphisms of $V_r^{(k)}$ that correspond to Dehn twists along $d_1,\ldots d_t$.
We define iteratively a basic sequence of automorphisms $\{\gamma _{L,n}, \phi _{L,n}\}$ (compare with Section 7.1 of \cite{KMimplicit} where one particular basic sequence of automorphisms is used), which is determined by a sequence of $t+q$-tuples $L=\{(p_{1,n},\dots,p_{t,n},m_{1,n},\ldots,m_{q,n})\}_{n=1}^{\infty}$

Let $$\phi _{L,0}=1$$

$$\gamma _{L,n}=\phi_{L,n-1}\delta _1^{m_{1,n}}\ldots \delta _q^{m_{q,n}}, n\geq 1$$

$$\phi _{L,n}=\gamma_{L,n}\beta _1^{p_{1,n}}\ldots \beta _t^{p_{t,n}},n\geq 1$$

Assuming generic families have already been constructed for $V_i^{(k)}$, $i<r$, and for every vertex group in levels $k'>k$, and that $\Theta _k$ is a family of growing powers of Dehn twists for edges on level $k$, set $\Psi(k')=\Psi(V_{M_{k'}}^{(k')})\Theta _{k'}$ for $k<k'\leq N$ (in other words the generic family for level $k'$ is the generic family of the last vertex group at that level) and set $\Psi(N+1)=\{1\}$. Let $\pi_k:G_k\rightarrow G_{k+1}$ be the canonical epimorphism. Let $\Sigma_r^{(k)}=\{\psi_{1}\cdots\psi_{r-1}|\psi_{i}\in\Psi(V_i^{(k)})\}$ be the collection of all compositions of generic solutions for previous vertex groups. We then say that $$\Psi(V_r^{(k)})=\{\mu _{L,n,\lambda_n}=\phi _{L,n}\delta _1^{\lambda_n}\ldots \delta _q^{\lambda_n}\sigma_n\pi_k\tau |\sigma_n\in\Sigma_r^{(k)}, \tau\in\Psi(k+1)\}_{n=1}^{\infty}$$ where each $\lambda_n$ is some positive integer, is a generic family for $V_r^{(k)}$ if it has the following property: Given any $n$ and any tuple of positive numbers $\overrightarrow{A}=(A_1,\ldots,A_{nt+nq+1})$ with $A_i<A_j$ for $i<j$, $\Psi$ contains a homomorphism $\mu_{n,L,\lambda_n}$ such that the tuple

\begin{align*}
\overrightarrow{L}_{n,r_n} &=(p_{1,1},\ldots ,p_{t,1}, m_{1,2},\ldots ,m_{q,2},\ldots, m_{1,n},\ldots ,m_{q,n}, p_{1,n},\ldots ,p_{t,n}, \lambda_n) \\
 &= (L_1,\ldots,L_{nt+nq+1})
\end{align*}grows faster than $\overrightarrow{A}$ in the sense that $L_1\geq A_1$ and $L_{i+1}-L_{i}\geq A_{i+1}-A_{i}$.

Finally we set $\Psi(S)=\Psi(V_{M_1}^{(1)})$ to be a generic family of solutions in $\Gamma$ (or $\Gamma\ast F(Y)$) for the NTQ system $S(X,A)=1$. Notice that $\Psi(S)$ $\Gamma$-discriminates $\Gamma_{R(S)}$.

\begin{example}
Consider the equation $[x,y]=[a,b]$ over a torsion-free hyperbolic group $\Gamma$. Define the following $\Gamma$-automorphisms of $\Gamma[X]=\Gamma*F(X)$:
$$\beta:x\rightarrow x, y\rightarrow xy$$
$$\delta:x\rightarrow yx, y\rightarrow y$$
Notice that big powers of these automorphisms produce big powers of elements. Let $n=2a$ for some positive integer $a$, and define:
$$\phi_{p,n}=\delta^{m_1}\beta^{p_1}\cdots\delta^{m_a}\beta^{p_a}$$
where $p=(p_1,m_1\ldots,p_a,m_a)$. Now take the solution $\tau$ of $[x,y]=[a,b]$ with $\tau(x)=a,\tau(y)=b$ and consider the family of mappings
$$\Psi=\{\mu_{p,n}=\phi_{p,n}\tau, p\in P_n\}_{n=1}^{\infty}$$ where for each $n$, $P_n$ is an infinite set of $n$-tuples of large natural numbers. Then $\Psi$ is a generic family of solutions in $\Gamma$ for $[x,y]=[a,b]$.
\end{example}

Finally, if $H$ is maximal $\Gamma$-limit quotient of $\Gamma_{R(S)}$ (there is no other $\Gamma$-limit quotient $H_1$ such that the map sending images of generators of $\Gamma_{R(S)}$ in $H_1$ into their images in $H$ can be extended to a proper homomorphism), then since a discriminating  family of homomorphisms for $H$ factors through some branch $b$ of $\mathcal{T}(S,\Gamma)$, we have $\phi_b:\Gamma_{R(S)}\rightarrow N_b$, where $N_b$ is the $\Gamma$-NTQ group corresponding to this branch,  with $\phi_b(\Gamma_{R(S)})=H$. 

%Each maximal $\Gamma$-limit quotient is isomorphic to $\Gamma_{R(S_1)}$,  with $S_1=1$ an irreducible system, so there is a (canonical) fundamental sequence, with automorphisms on the top level corresponding to a JSJ decomposition of $\Gamma_{R(S_1)}$, which factors through a $\Gamma$-NTQ groups $N$ (as in Section 7 of \cite{elemFree}). This gives the canonical embeddings of each $\Gamma$-limit group $L$ into some $\Gamma$-NTQ group, as mentioned in the introduction. We then define a generic family of homomorphisms for such a $\Gamma$-limit group $\Gamma_{R(S_1)}$ to be a generic family for a $\Gamma$-NTQ group $N$ into which it embeds canonically in this manner. 

 \section{ A cover of a canonical Hom-diagram}\label{section:can-construct}

In this section we will prove Theorem \ref{thm:can-construct} which we restate here.
\addtocounter{theorem}{-1}
\begin{theorem}Let $S(Z, A)=1$ be a finite system of equations over $\Gamma$. There is an algorithm to construct  a  cover  of a   canonical completed Hom-diagram   for  the coordinate group $\Gamma _{R(S)}$. 

 Corrective extensions that form the cover  are given by their finite presentations as $\Gamma$-NTQ groups. \end{theorem}

The proof  is based on the observation that branches of the tree ${\mathcal T}(S,\Gamma )$ which contain generic families of solutions of irreducible systems whose  coordinate groups are  maximal $\Gamma$-limit quotients of $\Gamma _{R(S)}$, correspond to corrective extensions of canonical NTQ systems. 

\subsection{Sol-maximal $\Gamma$-limit quotients}

 We define a partial order on the set of $\Gamma$-limit groups $\{H_{b}=\phi_{b}(\Gamma_{R(S)})\leq N_{b}\}$ over all branches $b$ of ${\mathcal T}(S,\Gamma )$, with $N_{b}$ the $\Gamma$-NTQ group corresponding to $b$, $\phi_{b}(\Gamma_{R(S)})\leq N_{b}$ as in Section \ref{section:reworking}. It's often notationally convenient to denote $H_{b},N_{b},\phi_{b}$ by $H_i,N_i,\phi_i$, where $i\in I$ indexes the branches of ${\mathcal T}(S,\Gamma )$.  For given elements $H_i,H_j$ we say that $H_j\leq _{Sol} H_i$  if   for every  homomorphism  $\psi_j: H_j\rightarrow \Gamma$  that factors through $N_j$  there exists  a homomorphism  $\psi_i: H_i\rightarrow \Gamma$ that factors through $N_i$ such that $\phi _i\psi _i= \phi _j\psi _j $.  In this case the  canonical map $\phi _j: \Gamma _{R(S)}\rightarrow H_j$ can be split  as $\tau\phi _i,$  where $\phi _i$ is the canonical homomorphism $\phi _i: \Gamma _{R(S)}\rightarrow H_i$ and $\tau$ is a $\Gamma$-epimorphism from $H_i$ to $H_j$. 
 
\begin{prop}\label{image:quotient}There is an algorithm to find all branches $b$ where $H_b$  is Sol-maximal. 
\end{prop}
\begin{proof}Let $\{\ell_1,\ldots,\ell_n\}$ generate $\Gamma_{R(S)}$, and for each branch $b$ in $\mathcal{T}(S,\Gamma)$, let $H_b=\phi_b(\Gamma_{R(S)})$ be the images of the homomorphisms into the corresponding $\Gamma$-NTQ groups. We proceed by considering the finite collection $\{H_b\}$, and the finite collection  $\{(H_i,H_j)\}$ of pairs
from $\{H_b\}$.

First we describe an algorithm which tests a pair $(H_1,H_2)$ if $H_2\leq_{Sol} H_1$. Let $N_1(X,A)=1$, $X=\{x_1,\ldots,x_r\}$, and $N_2(Y,A)=1$, $Y=\{y_1,\ldots,y_s\}$, be $\Gamma$-NTQ systems such that $N_1=\Gamma_{R(N_1)}$ and $N_2=\Gamma_{R(N_2)}$. There are words $\phi_1(\ell_i)=\upsilon_i(X,A)\in \Gamma_{R(N_1)}$ and $\phi_2(\ell_i)=\mu_i(Y,A)\in \Gamma_{R(N_2)}$ for $1\leq i \leq n$.

 Then the following first order formula is true in $\Gamma$, if and only if $H_2\leq_{Sol}H_1$:
$$\forall h_1,\ldots,h_s\exists g_1,\ldots,g_r
(N_2(h_1,\ldots,h_s,A)=1\rightarrow ({N}_1(g_1,\ldots,g_r,A)=1$$ $$
\\\wedge \upsilon_1(g_1,\ldots,g_r,A)=\mu_1(h_1,\ldots,h_s,A)
\\\wedge\ldots\wedge \upsilon_s(g_1,\ldots,g_r,A)=\mu_s(h_1 ,\ldots,h_s,A))).$$
\\
\\
In other words, for each $\Gamma$-homomorphism $\alpha:\Gamma_{R(N_2)}\rightarrow\Gamma$ there is a $\Gamma$-homomorphism $\beta:\Gamma_{R({N}_1)}\rightarrow\Gamma$ such that $\beta_{H_1}=\tau\alpha_{H_2}$ (here $\alpha$ and $\beta$ are restricted to $H_2$ and $H_1$. 

%Also, conversely if the formula is false, then there is no such $\Gamma$-homomorphism $\sigma$. Can we claim that if the formula is false and $H_2$ is in fact a quotient of $H_1$ there is another quotient $H_1'$ of $H_1$ (found by formula? or from generic family) for which the formula is true, so know $H_2$ is quotient of $H_1'$

Denote the above formula in $\Gamma$, as constructed from the pair $H_1,H_2$, by $\Psi(H_1,H_2)$. We now describe an algorithm to test if $\Psi(H_1,H_2)$ is true. Every conjunction of equations is equivalent to some single equation, so let $S_{\Psi}=S_{\Psi}(h_1,\ldots,h_s,g_1,\ldots,g_r)=1$ be an equation equivalent to\\ ${N}_1(g_1,\ldots,g_r,A)=1
\wedge \upsilon_1(g_1,\ldots,g_r,A)=\mu_1(h_1,\ldots,h_s,A)
\wedge\ldots\wedge\\ \upsilon_s(g_1,\ldots,g_r,A)=\mu_s(h_1 ,\ldots,h_s,A).$

By Proposition \ref{IFT}, $\Psi(H_1,H_2)$ is true in $\Gamma$ if and only if the corresponding equation $S_{\Psi}$ (with $h_1,\ldots,h_s$ considered as coefficients in $N_2^*$) has a solution in each $N_2^*$, where $N_2^*$ belongs to the (canonical) set of corrective extensions of $N_2$. To obtain each corrective extension, for some abelian vertex groups $A_i$ of $N_2$, abelian groups $\overline{A}_i$ with $A_i\leq_{f.i}\overline{A}_i$ are constructed exactly by adding particular roots to each $A_i$. Notice, that since $N_1$, $N_2$ are $\Gamma$-NTQ groups,  the free products at their bases are free products of $\Gamma\ast F(y_1,\ldots, y_m)\ast \Gamma ^{y_{m+1}}\ast\ldots\ast\Gamma ^{y_{m+k}}$, therefore a corrective extension does not extend the base group. However, Proposition \ref{IFT} does not directly imply what those particular roots are. To find which roots to add, and construct the corrective extensions and check them for solutions, we use Dahmani's construction (\cite{DisomRel}) of 'canonical representatives' for elements in $N_2$ (since it is toral relatively hyperbolic with parabolic subgroups given by the abelian vertex groups $A_1,\ldots,A_k$ of the NTQ structure) in the free product $F(\ell_1,\ldots,\ell_n,y_1,\ldots ,y_{m+k})*A_1*\cdots*A_k$ (Theorem 4.4 in \cite{DisomRel}). These canonical representatives give a disjunction of equations $S_{\Psi}^{(1)}\vee\cdots\vee S_{\Psi}^{(N)}$, which is equivalent to a single equation $S_{\Psi}^{can}$, over $F(\ell_1,\ldots,\ell_n,y_1,\ldots ,y_{m+k})*A_1*\cdots*A_k$.

Solutions for $S_{\Psi}$ exist in $N_2$ if and only if solutions for $S_{\Psi}^{can}$ exist in \newline $F(\ell_1,\ldots,\ell_n,y_1,\ldots ,y_{m+k})*A_1*\cdots*A_k$. Canonical representatives also exist for $N_2^*$ (toral relatively hyperbolic as well), and using those, the equation $S_{\Omega}$ over $N_2^*$ (with $h_1,\ldots,h_s$ considered as coefficients in $N_2^*$) induces the same equation $S_{\Omega}^{can}$ over a free product  $F(\ell_1,\ldots,\ell_n,y_1,\ldots ,y_{m+k})*\overline{A}_1*\cdots*\overline{A}_k$ with the same roots added to the same abelian subgroups. This is because Dahmani's construction of canonical representatives, applied to equivalent systems of triangular equations, give new equations corresponding to the ``middles'' of the triangles which, if the middles are ``short'' (as in hyperbolic case), are determined by $S_{\Omega}$, and if the middles belong to parabolic subgroups, are just commuting equations.  So similarly, solutions for $S_{\Omega}$ exist in $N_2^*$ if and only if solutions for $S_{\Psi}^{can}$ exist in $F(\ell_1,\ldots,\ell_n,y_1,\ldots ,y_{m+k})*\overline{A}_1*\cdots*\overline{A}_k$. It is these free product extensions which can be found and checked algorithmically. Indeed, Diekert and Lohrey proved \cite{DL}  that if the universal theories of the
factors $R_1$ and $R_2$ are decidable, so is the universal theory of $R_1*R_2$. In the abelian factors, conventional linear algebraic methods of finding solutions determine the finite index extensions (by giving which roots of elements are needed for solutions).  In $F(\ell_1,\ldots,\ell_n,y_1,\ldots ,y_{m+k})$ the universal theory is decidable by \cite{Mak84}.
\end{proof}
If $H_2\leq_{Sol}H_1$ we remove $H_2, N_2$ and the corresponding branch from ${\mathcal T}(S,\Gamma )$. So every non-Sol-maximal $\Gamma$-limit quotient will be removed. 

 We will also remove   $H_i, N_i$ if  in the diagram that we are considering there exists a family $\{H_j,N_j,j\in I\}$ such that for every  homomorphism  $\psi_i: H_i\rightarrow \Gamma$  that factors through $N_i$  there exists  for some $j\in J$ a homomorphism  $\psi_j: H_j\rightarrow \Gamma$ that factors through $N_j$ such that $\phi _i\psi _i= \phi _j\psi _j $. We call such a pair $H_i, N_i$ {\em redundant}.  After removing one redundant branch (pair  $H_i, N_i$) we consider the remaining diagram and identify redundant branches  in the new diagram. We keep removing branches until the obtained diagram does not have redundant branches.
 \begin{remark} \label{rk2} The algorithm to find these pairs  $H_i, N_i$ is very similar to the algorithm for removing non Sol-maximal quotients and their branches. 
Instead of the formula $\Psi(H_j,H_i)$ one has to consider the formula $\Psi((\vee _{j\in J}H_j),H_i)$, with the same premise of the implication but the conclusion being the disjunction  of the conclusions for $H_j,N_j, j\in J$.  But the  disjunction is equivalent to a system of equations in $\Gamma$. \end{remark}
\begin{definition}
Denote the subtree of  ${\mathcal T}(S,\Gamma )$ obtained after removal of these branches by  $\hat {\mathcal T}(S,\Gamma )$.
\end{definition}

\begin{prop}\label{construct:Grushko}Given a finite system of equations $S(Z,A)=1$ over $\Gamma$, there is an algorithm which finds the Grushko decomposition of of $\phi _b(\Gamma_{R(S)})$ for each branch $b$ of $\hat {\mathcal T}(S,\Gamma )$.
\end{prop}
\begin{proof}We show how to construct a free decomposition of the image $H$ of $\phi_b:\Gamma_{R(S)}\rightarrow N_b$, for each strict fundamental sequence $\Phi_b$.
If $H$ is a non-trivial free product  then there is a solution of the system $S(Z,A)=1$ in $\Gamma \ast\Gamma$. Therefore we can solve the corresponding systems  induced by canonical representatives of elements from 
$\Gamma \ast\Gamma$ in the group $F(A)\ast F(A)$. Then we will see the free product decomposition  of the corresponding ($\Gamma \ast\Gamma$)-NTQ group  (since generators of each copy of $\Gamma$ do not interact in the NTQ structure). The same ($\Gamma \ast\Gamma$)-NTQ group can be considered as a  $\Gamma$-NTQ group  if instead of the second copy of $\Gamma$ we take $\Gamma$ conjugated by a new element. Therefore for every branch $b_1$ of  $\hat {\mathcal T}(S,\Gamma )$ corresponding to $H\cong H_{b_1}$  there is a branch $b_2$  such that  the $\Gamma$-NTQ group for this branch corresponds to $H\cong H_{b_2}$, is freely decomposable and $H_{b_1}\leq _{Sol} H_{b_2}$. Since $H\cong H_b$ is a Sol-maximal $\Gamma$-limit quotient, the $\Gamma$-NTQ group corresponding to $H$ must be freely decomposable, and we will obtain 
 the induced free product decomposition of $H$ (given by generators in the free factors of the $\Gamma$-NTQ group) and this free product decomposition is non-trivial. 
\end{proof}
\subsection{Proof of Theorem 2}
All solutions of the system $S(Z,A)=1$ in $\Gamma\ast F(Y)$ factor through the tree $\hat {\mathcal T}(S,\Gamma )$ that can be algorithmically constructed by Lemma \ref{le:spleff} and Remark \ref{rk2}.  We know that there exists a canonical  completed
$Hom$-diagram for $S(Z,A)=1$.

We will show that a $\Gamma$-NTQ group corresponding to a branch $b$ of $\hat {\mathcal T}(S,\Gamma )$  is exactly a corrective extension of some canonical $\Gamma$-NTQ group for  $H=\phi_b(\Gamma_{R(S)})$  (that appear in the canonical  completed
$Hom$-diagram for $H$).  Since  all solutions of the system $S(Z,A)=1$  factor through the tree $\hat {\mathcal T}(S,\Gamma )$, this implies that  $\hat {\mathcal T}(S,\Gamma )$ is  a  cover  of a canonical completed Hom-diagram for  $\Gamma_{R(S)}.$

\begin{prop} \label{maxx} Let $b$ be a branch of $\hat {\mathcal T}(S,\Gamma )$ and $\bar S=1$ an irreducible system of equations with $\phi_b(\Gamma_{R(S)})=H\cong \Gamma_{R(\bar S)}$.  Then the $\Gamma$-NTQ group $N_b$ corresponding to the completed fundamental sequence for $b$, is a corrective extension of  a canonical $\Gamma$-NTQ group $N=\Gamma_{R(W)}$ of $H$ ($N$ appears in a completed  canonical Hom-diagram for $H$).
\end{prop} 
%Suppose  a  $\Gamma$-limit quotient $H$  of $\Gamma _{R(S)}$ corresponds to a branch $b_H$ of $\hat {\mathcal T}(S,\Gamma )$. Suppose a branch $b$ of $\hat{\mathcal T}(S,\Gamma )$  with NTQ group $M$ contains a generic family of homomorphisms  for a canonical $\Gamma$-NTQ group $N$ of $H$. And suppose that $H$ is isomorphic to the image of $\Gamma _{R(S)}$ in  $M$.  Then $M$ is  a corrective extension of $N$. 

\begin{proof}  $H$ is embedded into some canonical $\Gamma$-NTQ groups, each corresponding to some branch in a canonical completed $Hom$-diagram for $\Gamma _{R(\bar S)}$.  Since  all solutions of the system $S(Z,A)=1$  factor through the tree $\hat {\mathcal T}(S,\Gamma )$, generic families of homomorphisms for these canonical $\Gamma$-NTQ groups  also factor through 
this tree.  From the construction of generic families it follows that if a generic family is divided into a finite number of subfamilies, then at least one of these subfamilies is again a generic family. Therefore, for each  canonical $\Gamma$-NTQ group  of  $H$   there is a branch  of the tree $\hat{\mathcal T}(S,\Gamma )$  that contains a generic family of solutions  of the system of equations corresponding to this  group. 

Suppose first that $b$ itself contains a generic family for $W=1$, where $N=\Gamma_{R(W)}$.  (At the end we will show that otherwise this branch would be removed from ${\mathcal T}(S,\Gamma )$ when constructing $\hat {\mathcal T}(S,\Gamma )$.)
 
The $\Gamma$-NTQ group $N_b=\Gamma _{R(T)}$ corresponding to the branch $b$, is the coordinate group for some irreducible system of equations $T(X,Y)=1$.  Having $H\leq N_b$,  we include the set of generators of $H$ (denote it $X$) into the set of generators for $N_b$,    $N_b=\langle X,Y|T(X,Y)=1\rangle$.  By Proposition \ref{IFT},  $T(X,Y)=1$ has a formula solution $Y=f(X,Z)$ in a corrective extension $N_{corr}(X,Z)$ of $N$.  By \cite{Groves}, all solutions of  $T(X,Y)=1$ in $N_{corr}$ 
are contained in a finite number of fundamental sequences over a free product of a free group and subgroups of $N_{corr}$. By Lemma \ref{le:spleff}, each fundamental sequence consists of a monomorphism from $\Gamma _{R(T)}$ onto  a free product of a free group and subgroups of $N_{corr}$ followed by different embeddings of subgroups of $N_{corr}$  into $N_{corr}$ and specialization of variables of the free group. Indeed, there are no automorphisms involved into fundamental sequences because $\Gamma _{R(T)}$ does not have splittings relative to $N_{corr}$.  So $N_b=\Gamma _{R(T)}$ is a free product of a free group and subgroups of $N_{corr}$.  By the same Lemma \ref{le:spleff}, there are no other free factors in $N_b$ except those containing factors in the Grushko decomposition of $H$, which can be algorithmically found by Proposition \ref{construct:Grushko}, therefore $N_b\leq N_{corr}$.
Moreover, every non-cyclic factor in the Grushko decomposition of $H$ is a subgroup of a factor $\bar N_b$ in a free decomposition of $N_b$, and $\bar N_b$ is a subgroup of a factor in the Grushko decomposition of $N_{corr}$. The free group factor in the Grushko decompositions of $H, N_b$ and $N_{corr}$ is the same and we can move it to the base level of these groups.
Therefore we can consider each factor $ \bar N_b$ separately.  Instead we just assume that
$H$ and $N$ are freely indecomposable.

%   Let $N=\Gamma _{R(W)}$, $N=\langle X|W(X)=1\rangle $.

Let $N_{corr}=N_1>N_2>\ldots  >N_n$ be a sequence of $\Gamma$-NTQ groups corresponding to different levels of $N_{corr}$.
The graph of groups decomposition $\Delta _1$ corresponding to the top level $N_1$ may contain MQH subgroups, abelian vertex groups and  a free product of non-QH non-abelian vertex groups which is $N_2$.  Notice that $N_1$ and $H$ have the same MQH subgroups.

%\begin{lemma}   $N_b$  does not have an abelian splitting where $N$ and non-QH nonabelian vertex groups of $\Delta _1$ are elliptic.  \end{lemma}
%\begin{proof}  This follows from  Lemma \ref{le:spl} because $H\leq N$ and such a splitting would correspond to the top level of $N_b$.
 %\end{proof}

By Lemma \ref{le:qh} for each MQH subgroup $Q$ of $N_b$, the intersection  $H\cap Q$  is a finite index subgroup of $Q$.  $H\cap Q$ is also an MQH subgroup of $H$ and $N_{corr}$. This implies that $Q=N \cap Q$.  Therefore MQH subgroups of $N$ and $N_b$ coincide.  Abelian vertex groups of $H$ are conjugate into abelian vertex groups of $N_b$ as subgroups of finite index, therefore abelian vertex groups of $N_b$ are finite index subgroups in the abelian vertex groups of $N_{corr}$.  In this case we can replace the abelian vertex groups of $N_{corr}$ by their finite index subgroups that are  abelian vertex groups of $N_b$. Since they still contain abelian vertex groups of $H$ this results  in another  corrective extension of $N$ that  we still call $N_{corr}$. We redefine the NTQ structure of $N_b$ so that  all of these QH and abelian vertex groups are at the top level $N_{b1}$ of $N_b$. Since $N_b$ is a $\Gamma$-NTQ group, there is a retraction of $N_b$ to the second level $N_{b2}$.

On the second level, we have to show that the image $\bar H_2$ of $H$ in $N_2$ is isomorphic to its image $H_2$ in $N_{b2}$. Since a generic family of homomorphisms  from $N$ factors through $N_b$ and the splitting associated to the top levels $N_1$ and $N_{b1}$  are the same, the homomorphisms from $N_2$ for that generic family factor through $N_{b2}$. So  $\bar H_2$ is a  quotient of $H_2$. $\bar H_2$ is discriminated by homomorphisms which are minimal with respect to the  group of canonical automorphisms corresponding to a JSJ decomposition of $H$, which is equivalent to the splitting associated with $N_1$. So this is  the   same  group  as that of canonical automorphisms for the splitting associated to  $N_{b1}$.  

The group  $N_b$ was constructed  from an $F$-NTQ group $M$.  There is a branch corresponding to $M$ in a canonical Hom-diagram for a system of equations (over a free group) induced by canonical  representatives of solutions to $S=1$.   Since  a generic family of solutions for $W=1$ factors through $N_b$, canonical representatives of solutions in this generic family factor through the $F$-NTQ group $M$ from which $N_b$ was constructed.  Elements of this generic family  are elements in $F(y_1,\ldots , y_m)*\Gamma *\Gamma ^{y_{m+1}}*\ldots \Gamma ^{y_{m+k }}$. If $\Gamma$ is generated by $A$, then 
  canonical representatives for this family are elements in 
$F(y_1,\ldots , y_m)*F(A)*F(A)^{y_{m+1}}*\ldots *F(A)^{y_{m+k}}$. MQH and abelian subgroups of the top level  $N_{b1}$ were constructed from MQH and abelian subgroups of the top level $M_1$ of the NTQ group $M$ in the canonical completed $Hom$-diagram  for the system of equations over $F(y_1,\ldots , y_m, A)$ for these canonical representatives.
If homomorphisms discriminating $H_2$ are not minimal with respect to the group of canonical automorphisms  corresponding to a JSJ decomposition of $H$, then the homomorphisms discriminating  $N_{b2}$ are not minimal with respect to the group of canonical automorphisms  associated to the splitting of $N_{b1}$. However, in that case  homomorphisms discriminating  the second level $M_2$ of $M$  are not minimal with respect to the group of canonical automorphisms for the splitting associated to $M_1$.  This contradicts to the assumption that $M$ is canonical.  Therefore $\bar H_2=H_2.$

Now we have exactly the same situation for the second levels of $N_b$ and $N_{corr}$  as  for the top levels. Again,   by Proposition \ref{IFT}, $N_{b2}$ has a homomorphism  into the free product  of $F(Y)$ and subgroups of $N_2$,   this homomorphism is unique, and   $\bar H_2$  is a subgroup of $N_{b2}$,.

  Going inductively from  top to  bottom we compare all levels of $N_{corr}$ and $N_b$ and show that they are the same except that  some abelian vertex groups of $N_b$ are finite index subgroups of corresponding abelian vertex groups of $N_{corr}$. But these abelian subgroups of $N_b$ still contain corresponding abelian subgroups of $N$ as subgroups of finite index.  
  
Since the bottom groups of $N_b$ and $N_{corr}$ are the same,  this proves that $N_b$ is a corrective extension of $N$ and proves the proposition.

We proved that for all  branches $b$ that contain generic families for $W=1$ the groups $N_b$ are corrective extensions of $N$.  Therefore all solutions in the strict fundamental sequence corresponding to $N$ factor through these branches.  Indeed, a finite family of corrective extensions of $N$ either contains all solutions in the strict fundamental sequence corresponding to $N$ of the whole generic family of solutions  is missing. 
Those branches $b$ for which  $H$ is embedded into $N_b$ but there is no  canonical $\Gamma$-NTQ group for $H$ such that a generic family  of solutions for this group factors through $N_b$ would be removed from ${\mathcal T}(S,\Gamma )$ when constructing $\hat {\mathcal T}(S,\Gamma )$.
\end{proof}

Since all solutions of the system $S(Z,A)=1$ in $\Gamma\ast F(Y)$ factor through the tree $\hat {\mathcal T}(S,\Gamma )$ and each branch by Proposition \ref{maxx} corresponds to a corrective extension of a canonical $\Gamma$-NTQ group for some quotient of $\Gamma _{R(S)}$,  this tree is a cover of a  canonical completed Hom-diagram for $\Gamma_{R(S)}.$

This proves Theorem \ref{thm:can-construct}.

We would like to thank N. Touikan for the important comments which helped improving
the quality of the paper.

 \bibliography{can-hom-biblio}
\bibliographystyle{amsplain}

\addcontentsline{toc}{chapter}{\numberline{}Bibliography }
\end{document}